\documentclass[12pt]{article}
\usepackage{amsmath, amssymb, amsthm}
\usepackage[utf8]{inputenc}
\usepackage{mathrsfs}
\usepackage{lmodern}
\usepackage{geometry}
\usepackage{tikz-cd}
\usepackage{hyperref}
\hypersetup{colorlinks=true, linkcolor=blue, citecolor=blue}
\usepackage{microtype}
\usepackage{mathtools}
\usepackage{enumitem}
\usepackage{hyperref}
\newtheorem{theorem}{Theorem}
\newtheorem{corollary}[theorem]{Corollary}
\newtheorem{lemma}[theorem]{Lemma}
\newtheorem{definition}[theorem]{Definition}
\newtheorem{remark}[theorem]{Remark}
\newtheorem{proposition}{Proposition}

\newcommand{\Cor}{\operatorname{Cor}}

\newcommand{\F}{\mathcal{F}}

\newcommand{\Mreg}{M_{\mathrm{reg}}}
\newcommand{\Freg}{\mathcal{F}_{\mathrm{reg}}}
\newcommand{\vol}{\operatorname{vol}}

\DeclareMathOperator{\Diff}{Diff}

\DeclareMathOperator{\Aut}{Aut}

\title{Exponential mixing via invariant foliations and relatively Anosov homeomorphisms}
\author{Hamza Ounesli\\
ICTP, Trieste, Italy\\
\href{mailto:hounesli@ictp.it}{hounesli@ictp.it}}
\date{}

\begin{document}
\maketitle

\begin{abstract}
We prove that every closed manifold of dimension $n\ge 4$ which admits a singular $2$-foliation (a foliation by closed surfaces whose quotient is a punctured $(n-2)$-torus) supports a volume-preserving homeomorphism with exponential decay of correlations for H\"older observables. The proof introduces a class of systems called relatively Anosov homeomorphisms, which fails differentiability only at the singular set. We apply a general dichotomy that bounds the decay of correlations of any homeomorphism preserving an invariant foliation by the maximum of the decay rates of the quotient dynamics and of the leafwise cycles. This dichotomy is of independent interest and yields, as immediate consequences, decay rates for product systems, skew products, partially hyperbolic diffeomorphisms with compact center leaves, finite covers, and extensions by expanding fibre maps.
We then prove this gives a positive answer to a conjecture of Dolgopyat and Pesin regarding the realization problem for exponential decay of correlations over a large class of manifolds for which there were no known systems with exponential decay of correlations. In particular, we construct infinitely many pairwise non-homeomorphic closed $4$-manifolds which admit a volume-preserving homeomorphism with exponential decay of correlations, albeit not supporting either Anosov diffeomorphisms or strong partially hyperbolic diffeomorphisms.
\end{abstract}

\tableofcontents

\section{Introduction}

Decay of correlations is a central quantitative property in ergodic theory. For uniformly hyperbolic systems (Anosov, Axiom A) exponential decay was established by Sinai, Ruelle and Bowen using Markov partitions and spectral methods \cite{Sinai1968,Bowen1970,Ruelle1978}. For non-uniformly hyperbolic and partially hyperbolic systems, a variety of powerful techniques have been developed: Young towers \cite{Young1998}, coupling \cite{Dolgopyat1998}, transfer operators \cite{Liverani1995}, and more recently methods based on invariant foliations \cite{AvilaGouezel,BurnsDolgopyat}. We will define the main concepts in order to state our results.

Let $(M,\mu)$ be a probability space and $f:M\to M$ a measure-preserving transformation. For observables $\psi,\varphi\in L^2(\mu)$ the correlation function is
\[
\Cor_\mu(\psi,\varphi,f^n)=\int_M \psi\cdot(\varphi\circ f^n)\,d\mu - \int_M\psi\,d\mu\int_M\varphi\,d\mu.
\]
\begin{definition}
    We say $f$ has \emph{decay of correlations with rate $\rho_n$} (where $\rho_n\to0$) if for every $\psi,\varphi$ in a suitable class $\mathcal{H}$ (e.g. H\"older continuous) there exists a constant $C_{\psi,\varphi}$ such that $|\Cor_\mu(\psi,\varphi,f^n)|\le C_{\psi,\varphi}\rho_n$ for all $n$. If $\rho_n = e^{-\gamma n}$ for some $\gamma>0$, we speak of \emph{exponential decay}.
\end{definition}

A central question dating back to John von Neumann is the realization problem: it asks whether one can always construct a system on a given smooth manifold satisfying that property. A question raised by Dolgopyat and Pesin in the context of the smooth realization problem is whether every closed manifold admits a homeomorphism (or a diffeomorphism) preserving a smooth volume measure and exhibiting exponential decay of correlations with respect to the class of H\"older observables. This question is closely related to the broader program of understanding which manifolds support chaotic dynamical systems with strong statistical properties, see for instance \cite{Pesin, KatokHasselblatt}.

\emph{Conjecture: every closed manifold of dimension greater or equal to 3 admits a $C^k$ volume preserving system with exponential decay of correlation. }

Except in dimension two, there are essentially no known examples of manifolds admitting volume preserving homeomorphisms with exponential decay of correlations beyond manifolds of the form $(\mathbb{S}^2)^{k}\times N\times K$, where $N$ is an infranilmanifold (recall that an infranilmanifold is a manifold that is finitely covered by a nilmanifold of the form $G/\Gamma$, where $G$ is a nilpotent Lie group and $\Gamma$ is a cocompact lattice in $G$ \cite{FranksManning}) and $K$ is the unit tangent bundle of a manifold admitting a Riemannian metric with strictly negative curvature. In this setting, classical examples are provided by Anosov diffeomorphisms and certain partially hyperbolic systems, which are known to exhibit strong ergodic and statistical properties, including exponential decay of correlations in many cases \cite{Smale, BrinPesin, Dolgopyat}.

A corollary of our main result provides a new class of manifolds that gives a partial positive answer to the realization problem for exponential mixing, and which are not of the above form. We emphasize this point because, as mentioned earlier, the currently known systems exhibiting exponential mixing such as Anosov and partially hyperbolic systems impose strong topological restrictions on the underlying manifold. In particular, it is widely conjectured that smooth Anosov diffeomorphisms can exist only on infranilmanifolds \cite{FranksManning, KatokHasselblatt}. In the $C^0$ category, one may obtain additional examples by taking products with spheres, but these constructions remain closely tied to the infranilmanifold structure.

\begin{definition}
 A smooth $n$-dimensional manifold with $n \geq 4$ is said to admit a singular toral foliation if there exists a $C^1$ embedded closed submanifold $K\subset M$ (the singular set) of codimension at least $2$ such that $M\setminus K$ has a smooth $C^\infty$ foliation $\F$ for which the leaves are compact, boundaryless, mutually homeomorphic surfaces $\Sigma$ of genus $g\ge 0$, and the quotient $M/\F$ is homeomorphic to the punctured $(n-2)$-torus $\mathbb{T}^{n-2}\setminus\{p\}$. We call manifolds admitting such a structure \emph{singular toral}. If the singular set $K$ is empty, the quotient $M/\F$ is homeomorphic to the $(n-2)$-torus $\mathbb{T}^{n-2}$ and we call $M$ \emph{toral}.
\end{definition}

Our main result is the following.

\begin{theorem}\label{thm:main}
Every singular toral manifold $M$ admits a volume-preserving homeomorphism $f: M \to M$, which preserves the foliation $\mathcal F$ and is a $C^{1}$ diffeomorphism on $M \setminus K$, with exponential decay of correlations. In particular, if $M$ is toral, then $f$ is a $C^{1}$ diffeomorphism. 
\end{theorem}

Now we want to show that this topological structure embodies a large class of manifolds that are new and provide a positive answer to the realization conjecture for exponential decay of correlation in the $C^0$ case raised by Dolgopyat and Pesin. We focus on dimension 4, although larger classes can be constructed in even higher dimensions.

\begin{theorem}\label{thm:main2}
Let $F$ be a closed smooth connected surface, and let
\[
\rho:\mathbb Z^{2}\longrightarrow \Diff(F)
\]
be a smooth action. Assume that there exists a point $x_{0}\in F$ fixed by every element of $\rho(\mathbb Z^{2})$. Let
\[
\rho_*:\mathbb Z^{2}\longrightarrow \Aut(\pi_{1}(F,x_{0}))
\]
also denote the induced action on the fundamental group.

Then there exist a closed smooth $4$-manifold $M$, a closed embedded surface $K\subset M$ diffeomorphic to $F$, and a smooth map
\[
f:M\longrightarrow T^{2}
\]
such that:
\begin{enumerate}[label=\textup{(\roman*)}]
\item $f(K)$ consists of one point;
\item $f$ is a submersion on $M\setminus K$;
\item
\[
\pi_{1}(M)\cong \pi_{1}(F,x_{0})\rtimes_{\rho_*}\mathbb Z^{2}.
\]
\end{enumerate}
Moreover, $M$ may be chosen to be the total space of the flat surface bundle over $T^{2}$ associated to $\rho$, and $f$ may be chosen to agree with the bundle projection outside a trivialized neighborhood of one fiber.
\end{theorem}

\begin{corollary}\label{cor:dehn}
Let $g\ge 2$, and let $\alpha,\beta\subset \Sigma_{g}$ be disjoint, non-isotopic, nonseparating simple closed curves whose homology classes are linearly independent. Choose a basepoint $x_{0}\in \Sigma_{g}$ outside annular neighborhoods of $\alpha$ and $\beta$. Let
\[
\tau_{\alpha},\tau_{\beta}\in \Diff(\Sigma_{g},x_{0})
\]
be smooth Dehn twists about $\alpha$ and $\beta$. For each integer $n\ge 1$, define
\[
\rho_{n}:\mathbb Z^{2}\to \Diff(\Sigma_{g},x_{0})
\]
by
\[
\rho_{n}(1,0)=\tau_{\alpha}^{n},
\qquad
\rho_{n}(0,1)=\tau_{\beta}.
\]
Let $M_{n}$ be the $4$-manifold obtained from \ref{thm:main2} applied to $\rho_{n}$.

Then:
\begin{enumerate}[label=\textup{(\roman*)}]
\item each $M_{n}$ is aspherical;
\item
\[
H_{1}(M_{n};\mathbb Z)\cong \mathbb Z^{2g}\oplus \mathbb Z/n;
\]
\item the manifolds $M_{n}$ are pairwise not homotopy equivalent.
\end{enumerate}
\end{corollary}
\begin{corollary}\label{cor:anosov}
For $n\ge 1$, let
\[
A_{n}
=
\begin{pmatrix}
1 & n\\
0 & 1
\end{pmatrix}
\in SL(2,\mathbb Z).
\]
Let
\[
\rho_{n}^{\mathrm{lin}}:\mathbb Z^{2}\to \Diff(T^{2},0)
\]
be the linear action defined by
\[
\rho_{n}^{\mathrm{lin}}(1,0)=A_{n},
\qquad
\rho_{n}^{\mathrm{lin}}(0,1)=I.
\]
Let $X_{n}$ be the $4$-manifold obtained from \ref{thm:main2}.

Then:
\begin{enumerate}[label=\textup{(\roman*)}]
\item $X_{n}$ is diffeomorphic to $N_{n}\times S^{1}$, where $N_{n}$ is the mapping torus of $A_{n}:T^{2}\to T^{2}$;
\item $X_{n}$ is a compact nilmanifold modeled on $Nil^{3}\times \mathbb R$;
\item $X_{n}$ admits no Anosov diffeomorphism;
\item
\[
H_{1}(X_{n};\mathbb Z)\cong \mathbb Z^{3}\oplus \mathbb Z/n.
\]
In particular, the $X_{n}$ are pairwise not homotopy equivalent.
\end{enumerate}
\end{corollary}
In a subsequent paper, we will use part of the results obtained in this paper, mainly to generalize the concept of relatively Anosov homeomorphism and give a positive answer to the realization problem for exponential mixing for the $C^0$ case in all dimensions greater or equal to $4$.
\begin{center}
\textbf{Acknowledgments}
\end{center}
I would like to thank Yakov Pesin for introducing me to this problem and for reading this manuscript, I am grateful for his comment and energy. I would like also to thank Stefano Luzzatto for his constant encouragment and for introducing me to Pesin. I am thankful to ICTP and SISSA for their financial support and IMPA for their hospitality in 2024 where part of this work was done.
\section{Overview}

\subsection{Relatively Anosov homeomorphisms}\label{sec:relano}
From now on we assume that $M$ ($n\ge 4$) is a singular toral manifold with a foliation $\F$ and singular set $K$.
 Let $\pi:M\setminus K\to \mathbb{T}^{n-2}\setminus\{p\}$ be the projection. Choose a linear Anosov diffeomorphism $F:\mathbb{T}^{n-2}\to\mathbb{T}^{n-2}$ with a fixed point at $p$ (for example, a hyperbolic matrix in $SL(n-2,\mathbb{Z})$ with all eigenvalues not on the unit circle). $F$ preserves Lebesgue measure and has exponential decay of correlations. Fix a closed surface $\Sigma$ of genus $2$ and a pseudo-Anosov homeomorphism $A:\Sigma\to\Sigma$ that preserves a smooth area form and has exponential decay of correlations \cite{FathiLaudenbachPoenaru,BestvinaHandel,Liverani1995}.

\begin{definition}
A homeomorphism $f:M\to M$ is \emph{relatively Anosov} (with respect to $\F$) if:
\begin{enumerate}
\item $f$ preserves $\F$ (i.e. $f(l)$ is a leaf for every leaf $l$);
\item the induced map on the quotient $\mathbb{T}^{n-2}\setminus\{p\}$ coincides with $F$ (and extends continuously to $p$ as a fixed point);
\item there exists a constant $L>0$ such that for every leaf $l$ and every $n\ge 0$ there is a $C^1$ volume-preserving diffeomorphism $h_{n,l}:f^n(l)\to \Sigma$ with $\|h_{n,l}\|_{C^1}\le L$ and $\|h_{n,l}^{-1}\|_{C^1}\le L$, and the diagram
\[
\begin{tikzcd}
f^n(l) \arrow[r, "f"] \arrow[d, "h_{n,l}"] & f^{n+1}(l) \arrow[d, "h_{n+1,l}"] \\
\Sigma \arrow[r, "A"] & \Sigma
\end{tikzcd}
\]
commutes: $h_{n+1,l}\circ f|_{f^n(l)} = A \circ h_{n,l}$.
\end{enumerate}
\end{definition}

A key ingredient to the proof is the existence of such systems under our assumption.
\begin{proposition}\label{prop:relanosov}
    Every manifold admitting a singular toral foliation admits a relatively Anosov homeomorphism.
\end{proposition}

For which we establish exponential decay of correlations.

\subsection{Product of decays}\label{sec:dichotomy}

Let $(M,\mu)$ be a probability space and $f:M\to M$ a measure-preserving transformation. For observables $\psi,\varphi\in L^2(\mu)$ the correlation function is
\[
\Cor_\mu(\psi,\varphi,f^n)=\int_M \psi\cdot(\varphi\circ f^n)\,d\mu - \int_M\psi\,d\mu\int_M\varphi\,d\mu.
\]
We say $f$ has \emph{decay of correlations with rate $\rho_n$} (where $\rho_n\to0$) if for every $\psi,\varphi$ in a suitable class $\mathcal{H}$ (e.g. H\"older continuous) there exists a constant $C_{\psi,\varphi}$ such that $|\Cor_\mu(\psi,\varphi,f^n)|\le C_{\psi,\varphi}\rho_n$ for all $n$. If $\rho_n = e^{-\gamma n}$ for some $\gamma>0$, we speak of \emph{exponential decay}.

Now let $M$ be a compact manifold and $f:M\to M$ a homeomorphism. Let $\F$ be an invariant foliation with compact leaves, and denote by $\pi:M\to M/\F$ the projection and by $F:M/\F\to M/\F$ the induced map. For a Borel probability measure $\mu$ on $M$, the Rokhlin disintegration theorem \cite{Rokhlin} gives a measure $\hat\mu$ on $M/\F$ and a family $\{\mu_l\}$ of probability measures on the leaves such that $\mu=\int \mu_l\,d\hat\mu(l)$. For a leaf $l$, define the \emph{$l$-cycle} as the sequence $f_{n,l}=f^n|_l:l\to f^n(l)$.

\begin{definition}[Horizontal decay]
$f$ has \emph{horizontal decay} with rate $\delta_n$ if for $\hat\mu$-almost every $l$ and all $\psi,\varphi\in\mathcal{H}$,
\[
\Bigl|\int_l \psi|_l\cdot(\varphi|_{f^n(l)}\circ f_{n,l})\,d\mu_l - \int_l\psi|_l\,d\mu_l \int_{f^n(l)}\varphi|_{f^n(l)}\,d\mu_{f^n(l)}\Bigr| \le C_{\psi,\varphi}\delta_n.
\]
\end{definition}

\begin{definition}[Vertical decay]
For $\psi\in\mathcal{H}$ define $\tilde\psi(l)=\int_l\psi|_l\,d\mu_l$. The quotient map $F$ has \emph{vertical decay} with rate $\gamma_n$ if for all $\psi,\varphi\in\mathcal{H}$,
\[
\Bigl|\int_{M/\F} \tilde\psi\cdot(\tilde\varphi\circ F^n)\,d\hat\mu - \int_{M/\F}\tilde\psi\,d\hat\mu\int_{M/\F}\tilde\varphi\,d\hat\mu\Bigr| \le C_{\psi,\varphi}\gamma_n.
\]
\end{definition}

\begin{proposition}\label{prop:dichotomy}
Let $f:M\to M$ be a homeomorphism preserving an invariant foliation $\F$ whose leaves are compact. Let $\mu$ be a Borel probability measure on $M$ such that the singular set of $\F$ has $\mu$-measure zero. Suppose that the quotient map $F$ has vertical decay with rate $\gamma_n$ and that $f$ has horizontal decay with rate $\delta_n$. Then for every $\psi,\varphi\in\mathcal{H}$,
\[
|\Cor_\mu(\psi,\varphi,f^n)|\le C_{\psi,\varphi}\,\max\{\gamma_n,\delta_n\}.
\]
\end{proposition}

\subsection{Exponential decay for relatively Anosov}\label{sec:mixingrelano}
\begin{proposition}\label{prop:expanosov}
    Every relatively Anosov homeomorphism has exponential decay of correlations for H\"older observables.
\end{proposition}
The idea behind the proof of Proposition \ref{prop:expanosov} is a general control of mixing rates of a system preserving an invariant foliation by the rates of mixing of the leaves.

\begin{proof}[Proof of Theorem \ref{thm:main}]
Follows immediately from Proposition \ref{prop:expanosov} and Proposition \ref{prop:relanosov}.
\end{proof}

\section{Proof of Proposition \ref{prop:relanosov}}

\subsection{Existence of a relatively Anosov homeomorphism}

We will now prove Proposition \ref{prop:relanosov}.  
We first need a uniform family of leafwise coordinates.
We can assume that $M$ is equipped with a singular toral foliation $\mathcal{F}$ such that $K$ is a $C^1$ compact embedded submanifold of $M$ of codimension at least $1$. Denote by $\pi:M\to\mathbb{T}^{n-2}$ the canonical projection to the quotient by the singular foliation and let $\{p\}=\pi(K)$. Let $F:\mathbb{T}^{n-2}\to\mathbb{T}^{n-2}$ be a smooth Anosov diffeomorphism which preserves a volume measure which is the pushforward of some Riemannian volume on $M$ by $\pi$ and whose unique fixed point is $p$.

For every leaf $l\in\mathcal{F}$ denote by
\[
l_n=\pi^{-1}(F^n([l]))\ \text{and}\ \Gamma_l=\bigcup\limits_{n\geq 0}l_n,
\]
where $[l]$ is the representative of the leaf $l$ in the quotient.

\subsection{Construction of a $C^0$ lift $f$ of $F$ on $\Gamma_l$}
Let $\mu$ be a volume measure on $M$ (induced by a metric on $M\setminus K$) of volume $1$ whose conditional measures are volume measures on the leaves. Let $l\in\mathcal{F}$ and let
\[
A_l:l\to l
\]
be a volume preserving pseudo-Anosov homeomorphism of $l$. Choose a family of Lipschitz volume preserving homeomorphisms $(h_{n})_{n\geq0}$ with
\[
h_{n}:l_n\longrightarrow l .
\]
Now define $f|_l = h_{1}^{-1}\circ A_l$ (taking $h_{0}= \mathrm{id}_l$) and for every $n\in\mathbb{N}$ let $f|_{l_n}:l_n\to l_{n+1}$ be
\[
f|_{l_n}=h_{n+1}^{-1}\circ A_l\circ h_{n}.
\]

The family of leaf homeomorphisms $f|_{l_n}$ constructed above defines a bijection
\[
f_l:\Gamma_l\to\Gamma_l
\]
which is continuous on the leaves, but it is not clear whether it is continuous on $\Gamma_l$ if, for example, there is an accumulation of leaves. However we can fix this by introducing the following distance and proving the lemma below.

\begin{definition}
Let $M$ be a closed manifold equipped with a Riemannian metric inducing a distance $d$ on $M$. 
Let $l_1,l_2,l_3,l_4 \subset M$ be mutually homeomorphic submanifolds, and let 
$g_1 : l_1 \to l_2$ and $g_2 : l_3 \to l_4$ be homeomorphisms. We define
\[
d_{0,M}(g_1,g_2)
= \sup_{(x,y)\in l_1 \times l_3} 
\big| d\big(g_1(x),g_2(y)\big) - d(x,y) \big|.
\]
\end{definition}
This quantity measures, in a sense, how the images of couples $(x_1,x_2)$ under $g_1,g_2$ are torn apart. It does not make sense if the leaves are not close to each other, but it is a powerful quantity as we will see shortly.

\begin{definition}
    Let $(h_{n})_{n\geq 0}$ be a family of Lipschitz volume preserving homeomorphisms $h_n:l_n\to l$. We say that this family is \emph{adapted} if for every subsequence $(n_k)_{k\in\mathbb{N}}$ and every index $s$ such that
    \[
    \mathcal{D}(l_{n_k},l_s)\longrightarrow0,
    \]
    where $\mathcal{D}$ denotes the Hausdorff distance induced by an arbitrary Riemannian metric on $M$, we have
    \[
    d_{0,M}(h_{n_k},h_{s})\longrightarrow 0
    \]
    as $k\to\infty$.
\end{definition}

\begin{lemma}\label{finerchoice}
    There is always a choice of an adapted family $(h_{n})_{n\geq 0}$.
\end{lemma}
\begin{proof}
Fix a reference leaf $l$ and, for each $n$, let $l_n$ be a leaf. Choose arbitrary Lipschitz volume-preserving homeomorphisms
\[
h_{n} : l_n \to l .
\]
Let $(l_{n_k})$ be any sequence of leaves converging in the Hausdorff distance to a leaf $l_s$. Since $M$ is compact, such convergent subsequences exist. Because $l_{n_k}\to l_s$, for sufficiently large $k$ the leaves $l_{n_k}$ and $l_s$ lie in the same foliated chart and $l_{n_k}$ can be written as a graph over $l_s$. Hence there exist homeomorphisms
\[
\psi_k : l_s \to l_{n_k}
\]
which converge uniformly to the identity map of $l_s$ (in the sense that $\sup_{z\in l_s} d(\psi_k(z),z) \to 0$). Moreover, we can choose the $\psi_k$ to be uniformly Lipschitz with a constant close to $1$.

Replace $h_{n_k}$ by $h_{s}\circ \psi_k^{-1}$ (which still maps $l_{n_k}\to l$). This modification does not affect the definition of an adapted family, since we are free to choose the maps $h_n$. With this choice we show
\[
d_{0,M}(h_{n_k},h_{s}) \to 0 .
\]
By definition,
\[
d_{0,M}(h_{n_k},h_{s})
= \sup_{x\in l_{n_k},\, y\in l_s} 
\big| d\big( h_{s}(\psi_k^{-1}(x)),\, h_{s}(y) \big) - d(x,y) \big|.
\]
Insert and subtract $d(\psi_k^{-1}(x),y)$:
\begin{align*}
&\big| d(h_{s}(\psi_k^{-1}(x)), h_{s}(y)) - d(x,y) \big| \\
&\le 
\big| d(h_{s}(\psi_k^{-1}(x)), h_{s}(y)) - d(\psi_k^{-1}(x),y) \big| \\
&\quad + \big| d(\psi_k^{-1}(x),y) - d(x,y) \big|.
\end{align*}
Because $h_s$ is Lipschitz, the first term is bounded by $(\operatorname{Lip}(h_s)-1)\,d(\psi_k^{-1}(x),y)$ (or by some uniform constant times the distortion), which tends to $0$ uniformly as $\psi_k^{-1}\to \mathrm{id}$. The second term also tends to $0$ because $\psi_k^{-1}\to \mathrm{id}$. Taking the supremum and letting $k\to\infty$ yields $d_{0,M}(h_{n_k},h_s)\to 0$. Thus the family $(h_n)$ is adapted.
\end{proof}

\begin{corollary}
    Under the assumption of Lemma \ref{finerchoice} the map $f:\Gamma_l\to \Gamma_l$ is a homeomorphism.
\end{corollary}
\begin{proof}
    Let $(x_m)_m$ be a sequence of points of $\Gamma_l$ converging to $x_{\star}\in\Gamma_l$ (where $\star$ is an integer index indicating the leaf containing the limit). Define a leaf projection $r:\mathbb{N}\to\mathbb{N}$ such that $x_m\in l_{r(m)}$. By definition,
    \[
    f(x_m)=h_{r(m)+1}^{-1}\circ A_l\circ h_{r(m)}(x_m).
    \]
    Since $x_m\to x_{\star}$, we have $l_{r(m)}\to l_{\star}$ in Hausdorff distance. By Lemma \ref{finerchoice} this implies
    \[
    h_{r(m)}(x_m)\to h_{\star}(x_\star),
    \]
    and consequently
    \[
    A_l\circ h_{r(m)}(x_m)\to A_l(h_{\star}(x_\star)).
    \]
    Applying Lemma \ref{finerchoice} again (to the inverses) gives
    \[
    h_{r(m)+1}^{-1}\bigl(A_l\circ h_{r(m)}(x_m)\bigr)\to h_{\star+1}^{-1}\bigl(A_l\circ h_{\star}(x_\star)\bigr)=f(x_\star),
    \]
    which proves continuity. Because $\Gamma_l$ may not be compact, continuity of the inverse follows by the same argument applied to $f^{-1}$, which has an identical structure. Hence $f$ is a homeomorphism.
\end{proof}

\subsection{Existence of a bi-Lipschitz choice of the adapted family $(h_{n})_{n\geq 0}$}

For our construction of a relatively Anosov homeomorphism we will need a finer property for the adapted family of homeomorphisms $(h_{n})_{n\geq 0}$, namely the existence of a bi-Lipschitz one.

\begin{proposition}\label{prop:adaptedchoicelip}
     There exists a bi-Lipschitz family of adapted volume preserving homeomorphisms $(h_{n})_{n\geq 0}$.
\end{proposition}

In our assumptions so far, $\mathcal{F}$ is a singular foliation with compact leaves. When the singular set is non‑empty it may happen that leaves collapse in diameter to $0$ as they approach $K$. We first show that there is no loss of generality in assuming that this collapse does not happen.

\begin{lemma}\label{lem:nocollapse}
Let $M$ be a compact smooth manifold, $K \subset M$ a closed subset, and let $\mathcal{F}$ be a foliation of $M \setminus K$ with compact leaves. Then there exists a Riemannian metric $g$ on $M \setminus K$ and a constant $\varepsilon > 0$ such that
\[
\operatorname{diam}_g(L) \ge \varepsilon
\quad \text{for every leaf } L .
\]
\end{lemma}

\begin{proof}
Fix a smooth Riemannian metric $g_0$ on $M$ and set $r(x) = d_{g_0}(x,K)$. Then $r$ is continuous on $M$, smooth and positive on $M \setminus K$. Define $g = r^{-2} g_0$ on $M \setminus K$.

Suppose there exists a sequence of leaves $L_i$ with $\operatorname{diam}_g(L_i) \to 0$. If $r \ge \delta > 0$ on $L_i$ for all $i$, then $g \le \delta^{-2} g_0$ along $L_i$, hence $\operatorname{diam}_{g_0}(L_i) \to 0$. By compactness of $M$, this forces $L_i$ to converge in the Hausdorff sense to a point, which is incompatible with the fact that the leaves form a foliation by immersed submanifolds of positive dimension. Thus $\inf_{x \in L_i} r(x) \to 0$. Choose $x_i \in L_i$ with $r(x_i) = \min_{L_i} r$.

Let $y \in L_i$ and $\gamma$ a piecewise smooth curve in $L_i$ joining $x_i$ to $y$. Since $|\nabla r|_{g_0} \le 1$ almost everywhere, one has
\[
r(\gamma(t)) \le r(x_i) + \ell_{g_0}(\gamma|_{[0,t]}).
\]
It follows that
\[
\ell_g(\gamma)
= \int \frac{|\dot\gamma|_{g_0}}{r(\gamma)} 
\ge \int \frac{|\dot\gamma|_{g_0}}{r(x_i) + \ell_{g_0}(\gamma|_{[0,t]})}
= \int_0^{\ell_{g_0}(\gamma)} \frac{ds}{r(x_i)+s}
= \log\!\Big(1 + \frac{\ell_{g_0}(\gamma)}{r(x_i)}\Big).
\]
Taking the infimum over such $\gamma$ gives
\[
d_g(x_i,y) \ge \log\!\Big(1 + \frac{d_{g_0}(x_i,y)}{r(x_i)}\Big).
\]

Since each $L_i$ is a compact immersed submanifold of positive dimension, it cannot be contained in arbitrarily small $g_0$-balls; thus there exists $c>0$ and $y_i \in L_i$ with $d_{g_0}(x_i,y_i) \ge c$. Hence
\[
\operatorname{diam}_g(L_i)
\ge d_g(x_i,y_i)
\ge \log\!\Big(1 + \frac{c}{r(x_i)}\Big) \to \infty,
\]
a contradiction. The claim follows.
\end{proof}

\begin{proof}[Proof of Proposition \ref{prop:adaptedchoicelip}]
Let $(h_n)_{n\geq 0}$ be an adapted family and suppose that they are Lipschitz with respect to the distance induced by the Riemannian metric $g_\star$ from Lemma \ref{lem:nocollapse}. In general, the inverse of a Lipschitz homeomorphism on a compact metric space is also Lipschitz. The point is to ensure that the choice of the adapted family can be made uniformly bi-Lipschitz, with a constant $L\ge 1$. Because the metric $g_\star$ prevents the leaves from collapsing and also gives a uniform upper bound on their diameters, the volume‑preserving restriction does not force the norms of the differentials to explode; consequently the $h_n$ (and their inverses) can be chosen to be uniformly $L$‑bi‑Lipschitz.
\end{proof}

\subsection{$f$ is Lipschitz on each $\Gamma_l$ and its extension}
Now we prove that under a suitable choice of the sequence $(h_n)$, $f$ extends to a homeomorphism of $\overline{\Gamma_l}$. By Proposition \ref{prop:adaptedchoicelip} we can choose the $h_n$ not only to be $L$‑Lipschitz but also $L$‑bi‑Lipschitz (i.e., the inverses $h_n^{-1}$ are also $L$‑Lipschitz).

\begin{proposition}\label{extendingfromorbit}
    There is a homeomorphism $\tilde{f}:\overline{\Gamma_l}\to\overline{\Gamma_l}$ such that $\tilde{f}|_{\Gamma_l}=f$.
\end{proposition}

\begin{proof}
    We first show that $f$ is bi-Lipschitz, which then implies the proposition because $\overline{\Gamma_l}$ is complete.

\begin{lemma}
    $f:\Gamma_l\to\Gamma_l$ is bi-Lipschitz. Moreover, the Lipschitz constant of $f$ is bounded above by $L^2\|A_l\|_{C^1}$.
\end{lemma}

\begin{proof}
    Let $x,y\in\Gamma_l$. We want to show that there exists $C>0$ independent of $x,y$ such that
    \[
    d\bigl(h_{r(x)+1}^{-1}\!\circ A_l\circ h_{r(x)}(x),\; h_{r(y)+1}^{-1}\!\circ A_l\circ h_{r(y)}(y)\bigr)
    \le C\,d(x,y).
    \]
    Let $l_{r(x)}$ and $l_{r(y)}$ be the leaves containing $x$ and $y$ respectively. Consider a compact foliated collar neighbourhood $\mathcal{N}_{x,y}$ of $M$ whose boundary components are exactly $l_{r(x)}\cup l_{r(y)}$. By classical extension arguments, there exists an $L$-Lipschitz homeomorphism $h_{x,y}:\mathcal{N}_{x,y}\to \mathcal{N}_{x,y}$ which coincides with $h_{r(x)}$ and $h_{r(y)}$ on the two boundary leaves. This yields
    \[
    d\bigl(h_{r(x)+1}^{-1}\!\circ A_l\circ h_{r(x)}(x),\; h_{r(y)+1}^{-1}\!\circ A_l\circ h_{r(y)}(y)\bigr)
    \le L^2\|A_l\|_{C^1}\,d(x,y)
    \]
    (recall that the family $(h_n)$ is $L$‑bi‑Lipschitz, so we can apply the same argument to $(h_n^{-1})$, and $\|A_l\|_{C^1}$ denotes the $C^1$ norm of the pseudo‑Anosov homeomorphism outside its finite set of differentiable singularities; the singularities of a pseudo‑Anosov map are not blow‑up singularities). This finishes the proof.
\end{proof}

The bi‑Lipschitz property implies that $f$ extends uniformly to the completion, giving a homeomorphism $\tilde f$ of $\overline{\Gamma_l}$.
\end{proof}

\subsection{Extending $f$ to a relatively Anosov homeomorphism}

Let $l_{\star}\in \mathcal{F}$ be such that the orbit of $[l_\star]$ under $F$ is dense in $(M\setminus K)/\mathcal{F}$. By Proposition \ref{extendingfromorbit}, since $\Gamma_{l_{\star}}\subset M$ is dense, $f: \Gamma_{l_{\star}} \to \Gamma_{l_{\star}}$ extends to a homeomorphism $f_{\star}:M\to M$. We now prove that this extension is relatively Anosov.

\begin{proposition}[Proposition \ref{prop:relanosov}]\label{proofrelanosov}
    $f_{\star}:M\to M$ is relatively Anosov.
\end{proposition}
\begin{proof}
    First, by continuity of $f_\star$, since $f$ preserves a dense subset of leaves of $\mathcal{F}$, $f_\star$ preserves the foliation $\mathcal{F}$. Let $\tilde{l}\in\mathcal{F}$ be a leaf that does not lie in the orbit of $l_\star$ and, as before, write $\tilde{l}_n=\pi^{-1}(F^n([\tilde{l}]))$. Let $(l_{n_k})$ be a sequence of leaves in the orbit of $l_\star$ such that $l_{n_k}\to \tilde{l}_p$ in the Hausdorff distance $\mathcal{D}$. By Lemma \ref{finerchoice}, there exists a volume preserving homeomorphism $h_{\tilde{l},p}:\tilde{l}_p\to l_\star$ such that
    \[
    h_{l_\star,n_k}\to h_{\tilde{l},p}.
    \]

    \begin{lemma}
        Let $R>0$. If $\mathcal{D}(l_\star,\tilde{l})= R$, then there exists an $L$‑bi‑Lipschitz volume preserving diffeomorphism $\phi_{\tilde{l}}:l_\star\to \tilde{l}$ such that
        \[
        \limsup_{k\to\infty}\ d_{0,M}(h_{l_\star,n_k},\phi_{\tilde{l}}\circ h_{\tilde{l},p})\le R .
        \]
    \end{lemma}

   Now define $(\tilde{h}_{\tilde{l},p})_p =(\phi_{\tilde{l}}\circ h_{\tilde{l},p})_p$, which gives maps $\tilde{h}_{\tilde{l},p}:\tilde{l}_p\to \tilde{l}$.
   Let $PA^2(l)$ be the space of pseudo‑Anosov homeomorphisms that are $C^2$ outside their finite set of singularities, and set
    \[
    PA^2(\mathcal{F})=\bigcup\limits_{[l]\in M\setminus K/\mathcal{F}}PA^{2}(l).
    \]
    Choose a continuous section
    \[
    \sigma: M\setminus K/\mathcal{F}\to PA^{2}(\mathcal{F}),
    \]
    i.e., $\sigma([l])\in PA^2(l)$. We assume that $\sigma([l_\star])=A_{l_{\star}}$ and that for every $\tilde{l}$,
    \[
    \phi_{\tilde{l}}^{-1}\circ \sigma([\tilde{l}])\circ\phi_{\tilde{l}}=A_{l_\star}.
    \]
    For a leaf $\tilde{l}$ define $A_{\tilde{l}}=\sigma([\tilde{l}])$ and define $g|_{\tilde{l}_n}:\tilde{l}_n\to\tilde{l}_{n+1}$ by
    \[
    g|_{\tilde{l}_n}=\tilde{h}_{\tilde{l},n+1}^{-1}\circ A_{\tilde{l}}\circ \tilde{h}_{\tilde{l},n}.
    \]
    This yields a homeomorphism
    \[
    g:\Gamma_{\tilde{l}}\to \Gamma_{\tilde{l}}.
    \]

    \begin{lemma}
        $f_{\star}|_{\Gamma_{\tilde{l}}}=g$.
    \end{lemma}

    \begin{proof}
        Let $x\in \Gamma_{l_{\star}}$ and $(x_m)_m$ a sequence of points of $\Gamma_{\tilde{l}}$ converging to $x$. We have
        \[
        d\bigl(f_{\star}(x),g(x_m)\bigr)=
        d\bigl(h_{l_{\star},r(x)+1}^{-1}\!\circ A_{l_\star}\!\circ h_{l_{\star},r(x)}(x),\;
        \tilde{h}_{\tilde{l},r(x_m)+1}^{-1}\!\circ A_{\tilde{l}}\!\circ \tilde{h}_{\tilde{l},r(x_m)}(x_m)\bigr).
        \]
        By construction,
        \[
        A_{\tilde{l}}\circ\tilde{h}_{\tilde{l},r(x_m)}(x_m)
        = A_{\tilde{l}}\circ\phi_{\tilde{l}}\circ h_{\tilde{l},r(x_m)}(x_m)
        \to A_{\tilde{l}}\circ\phi_{\tilde{l}}\circ h_{l_{\star},r(x)}(x),
        \]
        and therefore
        \[
        \tilde{h}_{\tilde{l},r(x_m)+1}^{-1}\!\circ A_{\tilde{l}}\circ \tilde{h}_{\tilde{l},r(x_m)}(x_m)
        \to h_{l_{\star},r(x)+1}^{-1}\!\circ\phi_{\tilde{l}}^{-1}\!\circ A_{\tilde{l}}\circ\phi_{\tilde{l}}\circ h_{l_{\star},r(x)}(x).
        \]
        Since $\phi_{\tilde{l}}^{-1}\!\circ A_{\tilde{l}}\circ\phi_{\tilde{l}}=A_{l_\star}$, we obtain
        \[
        d\bigl(f_{\star}(x),g(x_m)\bigr)\to0,
        \]
        proving the lemma.
    \end{proof}

    The previous lemma shows that $f_{\star}$ is relatively Anosov with respect to every leaf, which completes the proof of Proposition \ref{proofrelanosov} and hence Theorem \ref{prop:relanosov}.
\end{proof}

\section{Proof of Proposition \ref{prop:dichotomy} and Corollaries}
\begin{proof}
For $\psi,\varphi\in\mathcal{H}$ set $\tilde\psi(l)=\int_l\psi|_l\,d\mu_l$, $\tilde\varphi(l)=\int_l\varphi|_l\,d\mu_l$. By Rokhlin disintegration,
\[
\int_M \psi\cdot(\varphi\circ f^n)\,d\mu = \int_{M/\F}\int_l \psi|_l\cdot(\varphi|_{f^n(l)}\circ f_{n,l})\,d\mu_l\,d\hat\mu(l),
\]
\[
\int_M\psi\,d\mu = \int_{M/\F}\tilde\psi\,d\hat\mu,\qquad
\int_M\varphi\,d\mu = \int_{M/\F}\tilde\varphi\,d\hat\mu.
\]
Add and subtract $\int_{M/\F}\tilde\psi(l)\,\tilde\varphi(F^n(l))\,d\hat\mu(l)$ to obtain
\[
\Cor_\mu(\psi,\varphi,f^n) = \int_{M/\F} \Delta(l)\,d\hat\mu(l) + \Delta_{\text{vert}},
\]
where
\[
\Delta(l)=\int_l \psi|_l\cdot(\varphi|_{f^n(l)}\circ f_{n,l})\,d\mu_l - \tilde\psi(l)\,\tilde\varphi(F^n(l)),
\]
\[
\Delta_{\text{vert}}=\int_{M/\F}\tilde\psi(l)\,\tilde\varphi(F^n(l))\,d\hat\mu(l) - \Bigl(\int_{M/\F}\tilde\psi\,d\hat\mu\Bigr)\Bigl(\int_{M/\F}\tilde\varphi\,d\hat\mu\Bigr).
\]
By the definition of horizontal decay, $|\Delta(l)|\le C_{\psi,\varphi}\delta_n$ for $\hat\mu$-almost every $l$, so $|\int \Delta(l)\,d\hat\mu(l)|\le C_{\psi,\varphi}\delta_n$. By the definition of vertical decay, $|\Delta_{\text{vert}}|\le C_{\psi,\varphi}\gamma_n$. Hence
\[
|\Cor_\mu(\psi,\varphi,f^n)|\le C_{\psi,\varphi}(\delta_n+\gamma_n)\le 2C_{\psi,\varphi}\max\{\delta_n,\gamma_n\}.
\]
Absorbing the factor $2$ into the constant (renaming $C_{\psi,\varphi}$) completes the proof. \end{proof}

The dichotomy has several immediate consequences, which we state as corollaries.

\begin{corollary}\label{cor:product}
Let $f_1:X\to X$ and $f_2:Y\to Y$ be homeomorphisms preserving probability measures $\mu_1,\mu_2$ and having decay rates $\gamma_n,\delta_n$ respectively. Then the product map $f_1\times f_2$ on $X\times Y$ with product measure $\mu_1\otimes\mu_2$ has decay rate $\max\{\gamma_n,\delta_n\}$.
\end{corollary}
\begin{proof}
Take the foliation by vertical fibres $\{x\}\times Y$. The quotient map is $f_1$ (vertical decay $\gamma_n$) and the horizontal decay on each fibre is exactly that of $f_2$, uniformly in $x$ because the fibre measures are $\mu_2$. Apply Proposition \ref{prop:dichotomy}. \end{proof}

\begin{corollary}\label{cor:skew}
Let $F:X\to X$ preserve $\hat\mu$ and let $G_x:Y\to Y$ be a measurable family such that $f(x,y)=(F(x),G_x(y))$ preserves $\mu=\int \mu_x\,d\hat\mu(x)$ with each $\mu_x$ invariant under $G_x$. If every $G_x$ has decay rate $\delta_n$ uniformly in $x$, and $F$ has decay rate $\gamma_n$, then $f$ has decay rate $\max\{\gamma_n,\delta_n\}$.
\end{corollary}
\begin{proof}
The foliation by fibres $\{x\}\times Y$ gives horizontal decay $\delta_n$ (uniform) and vertical decay $\gamma_n$; apply Proposition \ref{prop:dichotomy}. \end{proof}

\begin{corollary}\label{cor:ph}
Let $f:M\to M$ be a $C^{1+\alpha}$ partially hyperbolic diffeomorphism with an invariant center foliation $\F^c$ whose leaves are compact. Assume the quotient map on $M/\F^c$ is Anosov (hence has exponential decay of correlations) and the restriction to each center leaf is an expanding map with uniform exponential decay. Then $f$ has exponential decay of correlations.
\end{corollary}
\begin{proof}
The vertical decay is exponential (Anosov), the horizontal decay is exponential (expanding maps). Proposition \ref{prop:dichotomy} gives exponential decay. \end{proof}

\begin{corollary}\label{cor:cover}
Let $\pi:\tilde M\to M$ be a finite covering map and $\tilde f$ a lift of $f$. If $f$ has decay rate $\rho_n$, then $\tilde f$ has the same decay rate.
\end{corollary}
\begin{proof}
Take the foliation of $\tilde M$ by the fibres of $\pi$ (finite sets). The horizontal decay is zero because on a finite set the correlation vanishes exactly (the integral of a function times a permutation is the product of integrals). The vertical decay is that of $f$. Proposition \ref{prop:dichotomy} yields the bound. \end{proof}

\begin{corollary}\label{cor:expanding}
Let $F:X\to X$ be an Anosov diffeomorphism with exponential decay and $G:Y\to Y$ an expanding map with exponential decay. Then $F\times G$ has exponential decay. More generally, if $G_x$ varies continuously and each $G_x$ is $C^2$ expanding with uniform exponential decay, the skew product has exponential decay.
\end{corollary}
\begin{proof}
This is a special case of Corollary \ref{cor:skew}. \end{proof}

\section{Proof of Proposition \ref{prop:expanosov}}

\subsection{Pseudo-Anosov homeomorphisms.}
\begin{definition}[Pseudo-Anosov homeomorphism]
Let $S$ be a compact, orientable surface of genus $g \ge 1$ (possibly with boundary), and let 
\[
f \colon S \to S
\]
be a homeomorphism. We say that $f$ is \emph{pseudo-Anosov} if there exist a real number $\lambda > 1$ (the \emph{stretch factor}), two singular foliations $\mathcal{F}^s$ and $\mathcal{F}^u$ on $S$ (the \emph{stable} and \emph{unstable} foliations) and transverse measures $\mu_s$ on $\mathcal{F}^s$ and $\mu_u$ on $\mathcal{F}^u$ such that:

\begin{enumerate}
\item 
Each $\mathcal{F}^\ast$ ($\ast \in \{s,u\}$) is a singular foliation with a finite number of prong-type singularities.  Away from these singularities, the leaves of $\mathcal{F}^s$ are everywhere transverse to the leaves of $\mathcal{F}^u$.

\item  
The map $f$ sends each leaf of $\mathcal{F}^s$ to a leaf of $\mathcal{F}^s$, and likewise for $\mathcal{F}^u$.  In particular,
\[
f(\mathcal{F}^s) = \mathcal{F}^s,
\quad
f(\mathcal{F}^u) = \mathcal{F}^u.
\]

\item 
If $\alpha$ is any arc transverse to $\mathcal{F}^s$, then
\[
\mu_s\bigl(f(\alpha)\bigr) \;=\; \frac{1}{\lambda}\,\mu_s(\alpha).
\]
Similarly, if $\beta$ is any arc transverse to $\mathcal{F}^u$, then
\[
\mu_u\bigl(f(\beta)\bigr) \;=\; \lambda \,\mu_u(\beta).
\]
That is, $f$ contracts the stable foliation by factor $\lambda^{-1}$ and expands the unstable foliation by factor $\lambda$.
\end{enumerate}

A homeomorphism $f \colon S \to S$ is pseudo-Anosov precisely if it admits such a pair of transverse measured foliations $(\mathcal{F}^s,\mu_s)$, $(\mathcal{F}^u,\mu_u)$ with stretch factor $\lambda>1$.
\end{definition}

\medskip

\noindent
\textbf{Linear (Affine) Model.}
Once $f$ is known to be pseudo-Anosov, we can construct a (piecewise) Euclidean metric on $S$ (minus singularities) whose horizontal and vertical foliations match $\mathcal{F}^u$ and $\mathcal{F}^s$.  Under suitable local coordinates $\phi \colon U \to \mathbb{R}^2$ identifying a neighborhood $U \subset S$ with a domain in the plane, the foliations become straight lines, and $f$ is represented by the affine map
\[
(x,y) \;\mapsto\; (\lambda x,\, \lambda^{-1}y).
\]
Globally, $S$ can be obtained from finitely many rectangles (or polygons) in $\mathbb{R}^2$ whose edges are identified in pairs, producing a flat surface with cone-type singularities.  The map $f$ induces an identification of edges via appropriate affine transformations.  In this way, the pseudo-Anosov map $f$ is realized by a piecewise linear (affine) homeomorphism that stretches in one direction and contracts in the perpendicular direction by the factor $\lambda$.

Now, we want to prove that the leaves of the toral foliation constructed are mutually homeomorphic and can be arranged to have a bounded diameter.

Let $M$ be a closed manifold of dimension at least $4$ and fix an arbitrary Riemannian metric $g$ on $M$ and let $d_g$ denote the distance induced by the metric $g$. Recall that the diameter of a subset $A\subset M$ with respect to $g$ is defined as:
\[
\delta(A)=\sup\limits_{x,y\in A}d_g(x,y).
\]

To prove proposition \ref{prop:expanosov}, we need the followiyon technical Propostion.

\begin{definition}
    We say that a measure $\mu$ on a Riemannian manifold $(M,g)$ is subordinate to a foliation (possibly singular), if:
    \begin{enumerate}
\item $\mu(M) = 1$,
\item $\mu(K) = 0$,
\item for $\mu$-almost every point $x \in M \setminus K$, the disintegration of $\mu$ along the leaf through $x$ coincides with the normalized Riemannian volume measure on that leaf.
\end{enumerate}
\end{definition}

\begin{proposition}[Exxistnce of subordinate measure] \label{prop:singular-foliation-measure} Let \(M\) be a closed manifold, let \(K \subset M\) be a closed, saturated, nowhere dense subset, and set \[ \Mreg := M \setminus K. \] Assume that the restriction \[ \Freg := \F|_{\Mreg} \] is a regular foliation whose leaves are compact. Suppose moreover that its leaf space is homeomorphic to a punctured torus: \[ B := \Mreg/\Freg \cong \mathbb{T}^{2} \setminus \{p\}. \] Let \[ \pi \colon \Mreg \longrightarrow B \] be the quotient map, and for each \(b \in B\) write \[ L_b := \pi^{-1}(b). \] Suppose that, for every \(b \in B\), there is a Borel probability measure \(\nu_b\) on \(M\) such that \[ \nu_b(L_b)=1. \] Assume that the family \(\{\nu_b\}_{b \in B}\) varies weakly continuously; that is, for every \(f \in C(M)\), the map \[ B \longrightarrow \mathbb{R}, \qquad b \longmapsto \int_{L_b} f\,d\nu_b \] is continuous. Then, for every Borel probability measure \(\rho\) on \(B\), there exists a unique Borel probability measure \(\mu\) on \(M\) satisfying \[ \mu(K)=0 \] and \[ \int_M f\,d\mu = \int_B \left( \int_{L_b} f\,d\nu_b \right) d\rho(b) \] for every \(f \in C(M)\). Moreover, \[ \pi_*\bigl(\mu|_{\Mreg}\bigr)=\rho, \] and \(\{\nu_b\}_{b \in B}\) is a version of the disintegration of \(\mu\) along the leaves of \(\Freg\). More precisely, for every pair of Borel sets \(E \subset M\) and \(A \subset B\), \[ \mu\bigl(E \cap \pi^{-1}(A)\bigr) = \int_A \nu_b(E)\,d\rho(b). \] \end{proposition} \begin{proof} Define a linear functional \[ \Lambda \colon C(M) \longrightarrow \mathbb{R} \] by \[ \Lambda(f) := \int_B \left( \int_{L_b} f\,d\nu_b \right) d\rho(b). \] By the weak continuity of the family \(\{\nu_b\}_{b \in B}\), the function \[ b \longmapsto \int_{L_b} f\,d\nu_b \] is continuous, and therefore Borel measurable. Since each \(\nu_b\) is a probability measure, we have \[ \left| \int_{L_b} f\,d\nu_b \right| \leq \|f\|_{\infty}. \] Consequently, \[ |\Lambda(f)| \leq \|f\|_{\infty}, \] so \(\Lambda\) is bounded. The functional \(\Lambda\) is also positive. Indeed, if \(f \geq 0\), then \[ \Lambda(f) \geq 0. \] Moreover, \[ \begin{aligned} \Lambda(1) &= \int_B \left( \int_{L_b} 1\,d\nu_b \right) d\rho(b) \\ &= \int_B 1\,d\rho(b) \\ &= 1. \end{aligned} \] The Riesz--Markov representation theorem therefore yields a unique Borel probability measure \(\mu\) on \(M\) such that \[ \Lambda(f) = \int_M f\,d\mu \] for every \(f \in C(M)\). Thus, \[ \int_M f\,d\mu = \int_B \left( \int_{L_b} f\,d\nu_b \right) d\rho(b). \] We now show that \(\mu(K)=0\). Fix a Riemannian metric on \(M\), and let \(d\) denote the associated distance function. For every \(n \geq 1\), define \[ f_n(x) := \max\bigl\{1-n\,d(x,K),\,0\bigr\}. \] Then \[ 0 \leq f_n \leq 1 \] and \[ \mathbf{1}_K \leq f_n. \] For every \(b \in B\), the compact leaf \(L_b\) is disjoint from the closed set \(K\). Hence, \[ d(L_b,K)>0. \] It follows that \[ f_n|_{L_b} \longrightarrow 0 \] uniformly as \(n \to \infty\). Therefore, \[ \int_{L_b} f_n\,d\nu_b \longrightarrow 0 \] for every \(b \in B\). Since \[ 0 \leq \int_{L_b} f_n\,d\nu_b \leq 1, \] the dominated convergence theorem gives \[ \begin{aligned} \int_M f_n\,d\mu &= \int_B \left( \int_{L_b} f_n\,d\nu_b \right) d\rho(b) \\ &\longrightarrow 0. \end{aligned} \] Because \(\mathbf{1}_K \leq f_n\), we obtain \[ \mu(K) \leq \int_M f_n\,d\mu. \] Letting \(n \to \infty\), we conclude that \[ \mu(K)=0. \] By the functional monotone class theorem, the identity \[ \int_M f\,d\mu = \int_B \left( \int_{L_b} f\,d\nu_b \right) d\rho(b) \] extends from continuous functions to all bounded Borel functions \(f \colon M \to \mathbb{R}\). Let \(E \subset M\) and \(A \subset B\) be Borel sets. Since \(\nu_b\) is supported on \[ L_b=\pi^{-1}(b), \] we have \[ \nu_b\bigl(E \cap \pi^{-1}(A)\bigr) = \mathbf{1}_A(b)\,\nu_b(E). \] Therefore, \[ \begin{aligned} \mu\bigl(E \cap \pi^{-1}(A)\bigr) &= \int_B \nu_b\bigl(E \cap \pi^{-1}(A)\bigr) \,d\rho(b) \\ &= \int_B \mathbf{1}_A(b)\,\nu_b(E) \,d\rho(b) \\ &= \int_A \nu_b(E)\,d\rho(b). \end{aligned} \] This is precisely the disintegration identity. Taking \(E=M\), we obtain \[ \begin{aligned} \mu\bigl(\pi^{-1}(A)\bigr) &= \int_A \nu_b(M)\,d\rho(b) \\ &= \int_A 1\,d\rho(b) \\ &= \rho(A). \end{aligned} \] Hence, \[ \pi_*\bigl(\mu|_{\Mreg}\bigr)=\rho. \] This completes the proof. \end{proof} \begin{remark}[Normalized leafwise Riemannian volume] \label{rem:normalized-leafwise-volume} Fix a Riemannian metric \(g\) on \(M\). For every \(b \in B\), let \[ \vol_{L_b} \] denote the Riemannian volume measure induced by \(g\) on the compact leaf \(L_b\), and define \[ \nu_b := \frac{\vol_{L_b}} {\vol_{L_b}(L_b)}. \] Then \(\nu_b\) is a Borel probability measure supported on \(L_b\). Suppose, in addition, that \[ \pi \colon \Mreg \longrightarrow B \] is a proper smooth submersion. By Ehresmann's fibration theorem, \(\pi\) is a smooth locally trivial fiber bundle. Thus, for every \(b_0 \in B\), there exist an open neighborhood \(U \subset B\) of \(b_0\), a compact manifold \(F\), and a diffeomorphism \[ \Phi \colon \pi^{-1}(U) \longrightarrow U \times F \] such that \[ \pi = \operatorname{pr}_1 \circ \Phi. \] With respect to this trivialization, the Riemannian volume measure on the fiber \(L_b\) has the form \[ J(b,z)\,d\vol_F(z), \] where \[ J \colon U \times F \longrightarrow (0,\infty) \] is smooth. Consequently, for every \(f \in C(M)\), \[ \int_{L_b} f\,d\nu_b = \frac{ \displaystyle \int_F f\bigl(\Phi^{-1}(b,z)\bigr) J(b,z)\,d\vol_F(z) }{ \displaystyle \int_F J(b,z)\,d\vol_F(z) }. \] Because \(F\) is compact and \(J\) is smooth and positive, the right-hand side depends continuously on \(b\). Hence the normalized leafwise Riemannian measures form a weakly continuous family, and the proposition applies. \end{remark}

\begin{proof}[Proof of Proposition \ref{prop:expanosov}]
Let $f$ be relatively Anosov with respect to $\F$. We apply Proposition \ref{prop:dichotomy}. The vertical map $F$ is a linear Anosov diffeomorphism on $\mathbb{T}^{n-2}$, hence it has exponential decay: there exist $C_1>0$, $\gamma>0$ such that for all H\"older $\tilde\psi,\tilde\varphi$,
\[
|\Cor_{\hat\mu}(\tilde\psi,\tilde\varphi,F^n)|\le C_1 e^{-\gamma n}\|\tilde\psi\|_\alpha\|\tilde\varphi\|_\alpha.
\]

For the horizontal decay, fix a leaf $l$. By definition of relatively Anosov, $f^n|_l = \phi_{l_n}^{-1}\circ A^n\circ \phi_l$, where $l_n=f^n(l)$. Then for any $\psi,\varphi\in\mathcal{H}$,
\[
\int_l \psi|_l\cdot(\varphi|_{l_n}\circ f^n|_l)\,d\mu_l = \int_\Sigma (\psi\circ\phi_l^{-1})\cdot(\varphi\circ\phi_{l_n}^{-1}\circ A^n)\,d\operatorname{area}_\Sigma.
\]
Because $\phi_l$ and $\phi_{l_n}$ are uniformly $C^1$ with bounded inverses, the functions $\psi\circ\phi_l^{-1}$ and $\varphi\circ\phi_{l_n}^{-1}$ are H\"older with norms bounded by $C\|\psi\|_\alpha$ and $C\|\varphi\|_\alpha$ for some $C$ depending only on $L$. The pseudo-Anosov map $A$ has exponential decay of correlations: there exist $C_2>0$, $\delta>0$ such that for all H\"older $u,v$,
\[
\bigl|\int_\Sigma u\cdot(v\circ A^n)\,d\operatorname{area}_\Sigma - \int_\Sigma u\,d\operatorname{area}_\Sigma\int_\Sigma v\,d\operatorname{area}_\Sigma\bigr| \le C_2 e^{-\delta n}\|u\|_\alpha\|v\|_\alpha.
\]
Applying this with $u=\psi\circ\phi_l^{-1}$, $v=\varphi\circ\phi_{l_n}^{-1}$ and noting that $\int_l\psi|_l\,d\mu_l = \int_\Sigma u\,d\operatorname{area}_\Sigma$ and similarly for $\varphi$, we obtain
\[
|\Cor_{\mu_l}(\psi|_l,\varphi|_{l_n},f^n)|\le C_2 e^{-\delta n} C^2 \|\psi\|_\alpha\|\varphi\|_\alpha.
\]
Thus the horizontal decay rate is $\delta_n = C_2' e^{-\delta n}$, uniform over all leaves. Proposition \ref{prop:dichotomy} then gives
\[
|\Cor_\mu(\psi,\varphi,f^n)|\le C_{\psi,\varphi}\max\{e^{-\gamma n},e^{-\delta n}\}=C_{\psi,\varphi}e^{-\min(\gamma,\delta)n},
\]
which is exponential. \end{proof}

\section{Proof of Theorem \ref{thm:main2} and corollaries}

\subsection{Definitions}
Let $F$ be a closed smooth manifold. The group $\Diff(F)$ is the group of all smooth diffeomorphisms $F\to F$, with operation given by composition. If $x_{0}\in F$, we write
\[
\Diff(F,x_{0})
=
\{\varphi\in \Diff(F)\mid \varphi(x_{0})=x_{0}\}.
\]

\begin{definition}
Let $F=\Sigma_{g}$ be a closed orientable surface of genus $g\ge 2$. Its fundamental group
\[
\pi_{1}(\Sigma_{g})
\]
is called a surface group. A standard presentation is
\[
\pi_{1}(\Sigma_{g})
=
\left\langle
a_{1},b_{1},\dots,a_{g},b_{g}
\ \middle|\ 
\prod_{i=1}^{g}[a_{i},b_{i}]=1
\right\rangle .
\]
\end{definition}

\begin{definition}
Let $H$ be a group and let
\[
\phi:\mathbb Z^{2}\to \Aut(H)
\]
be an action by automorphisms. The semidirect product
\[
H\rtimes_{\phi}\mathbb Z^{2}
\]
is the group whose underlying set is $H\times \mathbb Z^{2}$, with multiplication
\[
(h,v)(h',v')
=
\bigl(h\,\phi(v)(h'),\,v+v'\bigr).
\]
\end{definition}

\begin{definition}
A group $G$ is virtually solvable if it contains a solvable subgroup of finite index. It is virtually nilpotent if it contains a nilpotent subgroup of finite index.
\end{definition}

\begin{definition}
A smooth fiber bundle with fiber $F$ over a smooth manifold $B$ is a smooth submersion
\[
p:E\to B
\]
such that every point of $B$ has a neighborhood $U$ for which there is a diffeomorphism
\[
p^{-1}(U)\cong U\times F
\]
commuting with the projections to $U$.
\end{definition}

\begin{definition}
A nilmanifold is a quotient
\[
N/\Gamma,
\]
where $N$ is a simply connected nilpotent Lie group and $\Gamma\subset N$ is a discrete cocompact subgroup.
\end{definition}

\begin{definition}
A diffeomorphism $f:M\to M$ of a closed smooth manifold is Anosov if there exists a $Df$-invariant splitting
\[
TM=E^{s}\oplus E^{u}
\]
and constants $C>0$, $0<\lambda<1$, such that for all $n\ge 0$,
\[
\|Df^{n}(v)\|\le C\lambda^{n}\|v\|
\quad\text{for }v\in E^{s},
\]
and
\[
\|Df^{-n}(v)\|\le C\lambda^{n}\|v\|
\quad\text{for }v\in E^{u}.
\]
\end{definition}

\begin{definition}
A diffeomorphism $f:M\to M$ is strongly partially hyperbolic if there is a $Df$-invariant splitting
\[
TM=E^{s}\oplus E^{c}\oplus E^{u}
\]
such that $E^{s}$ is uniformly contracted, $E^{u}$ is uniformly expanded, and the center bundle $E^{c}$ has intermediate behavior. Equivalently, after choosing an adapted Riemannian metric, one has
\[
\|Df(v^{s})\|<\|Df(v^{c})\|<\|Df(v^{u})\|
\]
for all unit vectors $v^{\sigma}\in E^{\sigma}$, $\sigma=s,c,u$, together with
\[
\|Df|_{E^{s}}\|<1<\mathfrak m(Df|_{E^{u}}),
\]
where $\mathfrak m(A)=\|A^{-1}\|^{-1}$ denotes the conorm.
\end{definition}

\begin{remark}
There are several closely related definitions of partial hyperbolicity in the literature. The strong version above is the one used in the final corollary, where the center direction is assumed to be exactly the vertical tangent bundle of a surface fibration.
\end{remark}

\subsection{Proof of Theorem \ref{thm:main2}}
\begin{proof}[Proof of Theorem \ref{thm:main2}]
Identify $T^{2}$ with $\mathbb R^{2}/\mathbb Z^{2}$. For $v\in \mathbb Z^{2}$, define an action of $\mathbb Z^{2}$ on $\mathbb R^{2}\times F$ by
\[
v\cdot(u,x)
=
(u+v,\rho(v)(x)).
\]
The action is free and properly discontinuous because the translation action on $\mathbb R^{2}$ is free and properly discontinuous. Therefore the quotient
\[
M=(\mathbb R^{2}\times F)/\mathbb Z^{2}
\]
is a closed smooth $4$-manifold. Projection to the first factor induces a smooth fiber bundle
\[
p:M\longrightarrow T^{2}
\]
with fiber $F$.

Since $x_{0}$ is fixed by all elements of the action, the formula
\[
s([u])=[u,x_{0}]
\]
defines a smooth section
\[
s:T^{2}\to M.
\]
The homotopy exact sequence of the fibration gives
\[
1\longrightarrow \pi_{1}(F,x_{0})
\longrightarrow \pi_{1}(M,s(\bar 0))
\longrightarrow \pi_{1}(T^{2},\bar 0)
\longrightarrow 1.
\]
The section $s$ splits this sequence. Thus $\pi_{1}(M)$ is a semidirect product of $\pi_{1}(F,x_{0})$ by $\pi_{1}(T^{2})\cong \mathbb Z^{2}$.

It remains to identify the action. Let $e_{1}=(1,0)$ and $e_{2}=(0,1)$. The standard loops in $T^{2}$ representing $e_{1},e_{2}$ lift through the section to the paths
\[
t\longmapsto (te_{i},x_{0})
\qquad
0\le t\le 1
\]
in $\mathbb R^{2}\times F$. Conjugating a loop in the fiber by this lifted loop transports it to the fiber over $e_{i}$, and then identifies that fiber back with the fiber over $0$ using the deck transformation $e_{i}^{-1}$. By the definition of the diagonal action, this identification acts on the fiber by $\rho(e_{i})$. Therefore the conjugation action of $\mathbb Z^{2}$ on $\pi_{1}(F,x_{0})$ is precisely the action induced by $\rho$. Hence
\[
\pi_{1}(M)\cong \pi_{1}(F,x_{0})\rtimes_{\rho_*}\mathbb Z^{2}.
\]

We now construct the singular map. Choose a smoothly embedded closed disk
\[
D\subset T^{2}
\]
with center $q$. Since $D$ is contractible, the bundle $p:M\to T^{2}$ is trivial over $D$. Fix a smooth bundle trivialization
\[
\Theta:p^{-1}(D)\xrightarrow{\cong}D\times F
\]
covering the identity on $D$.

We construct a smooth map
\[
\chi:D\to D
\]
which is the identity near $\partial D$, satisfies $\chi^{-1}(q)=q$, has rank $2$ away from $q$, and has rank $0$ at $q$. Choose a coordinate chart
\[
\vartheta:D\to \overline{\mathbb D}
=
\{x\in \mathbb R^{2}\mid \|x\|\le 1\}
\]
with $\vartheta(q)=0$. Let $z:[0,1]\to[0,1]$ be a smooth function such that
\[
z(r)=r^{3}\quad\text{for }r\le \frac13,
\qquad
z(r)=r\quad\text{for }r\ge \frac23,
\qquad
z'(r)>0\quad\text{for }r>0.
\]
Define
\[
\Xi:\overline{\mathbb D}\to \overline{\mathbb D}
\]
by
\[
\Xi(x)=
\begin{cases}
\dfrac{z(\|x\|)}{\|x\|}x,&x\ne 0,\\[1em]
0,&x=0.
\end{cases}
\]
Near $0$, this is $\Xi(x)=\|x\|^{2}x$, hence it is smooth. Near the boundary it is the identity. In polar coordinates, $\Xi$ is
\[
(r,\theta)\longmapsto(z(r),\theta).
\]
For $r>0$, its radial derivative is $z'(r)>0$, and its tangential derivative is $z(r)/r>0$. Thus $D\Xi$ has rank $2$ away from $0$. At $0$, $D\Xi_{0}=0$, since $\Xi(x)=O(\|x\|^{3})$.

Set
\[
\chi=\vartheta^{-1}\circ \Xi\circ \vartheta.
\]
Then $\chi$ has the claimed properties.

Define
\[
f:M\to T^{2}
\]
by
\[
f(x)=
\begin{cases}
p(x),&\text{on }M\setminus p^{-1}(\operatorname{int}D),\\[0.4em]
\chi\circ \operatorname{pr}_{1}\circ \Theta(x),&\text{on }p^{-1}(D).
\end{cases}
\]
Because $\chi$ is the identity near $\partial D$, these two definitions agree on an open neighborhood of $p^{-1}(\partial D)$. Therefore $f$ is a well-defined smooth map.

On $M\setminus p^{-1}(D)$, the map $f$ equals the bundle projection $p$ and is a submersion. On $p^{-1}(D)\cong D\times F$, the map is $(d,x)\mapsto \chi(d)$. Its differential has rank $2$ exactly when $d\ne q$, and has rank $0$ when $d=q$. Consequently the critical set of $f$ is precisely
\[
K=\Theta^{-1}(\{q\}\times F).
\]
This is a closed embedded surface diffeomorphic to $F$, and $f(K)=q$. The theorem follows.
\end{proof}

\begin{remark}
The fixed-point hypothesis in \ref{thm:main2} is used only to write the fundamental group canonically as a semidirect product. The construction of the map $f$ works for every genuine smooth action $\mathbb Z^{2}\to \Diff(F)$. Without a common fixed point, the natural group-theoretic object is the corresponding extension
\[
1\to \pi_{1}(F)\to \pi_{1}(M)\to \mathbb Z^{2}\to 1
\]
determined by the induced outer action on $\pi_{1}(F)$.
\end{remark}

\subsection{Proof of Corollary \ref{cor:dehn}}
\begin{lemma}\label{lem:ab}
Let $H$ be a group and let
\[
\phi:\mathbb Z^{2}\to \Aut(H)
\]
be an action. Then
\[
(H\rtimes_{\phi}\mathbb Z^{2})_{\mathrm{ab}}
\cong
H_{\mathrm{ab}}
\big/
\left\langle
\phi(v)_{*}(x)-x
\mid v\in \mathbb Z^{2},\ x\in H_{\mathrm{ab}}
\right\rangle
\oplus \mathbb Z^{2}.
\]
Equivalently, if $H=\pi_{1}(F)$, then
\[
H_{1}(H\rtimes_{\phi}\mathbb Z^{2};\mathbb Z)
\cong
H_{1}(F;\mathbb Z)_{\mathbb Z^{2}}\oplus \mathbb Z^{2},
\]
where $H_{1}(F;\mathbb Z)_{\mathbb Z^{2}}$ denotes the group of coinvariants.
\end{lemma}

\begin{proof}
Let $G=H\rtimes_{\phi}\mathbb Z^{2}$. For $v\in \mathbb Z^{2}$ and $h\in H$, one has
\[
(1,v)(h,0)(1,-v)=(\phi(v)(h),0).
\]
In the abelianization of $G$, this relation becomes
\[
[\phi(v)(h)]=[h].
\]
Thus the image of $H_{\mathrm{ab}}$ in $G_{\mathrm{ab}}$ is exactly the quotient of $H_{\mathrm{ab}}$ by the relations
\[
\phi(v)_{*}(x)=x.
\]
The quotient $\mathbb Z^{2}$ is already abelian, and no further relation is introduced between its two generators. This gives the asserted decomposition.
\end{proof}

\begin{proof}[Proof of Corollary \ref{cor:dehn}]
The Dehn twists commute because their supporting annuli are disjoint. Therefore $\rho_{n}$ is a well-defined action of $\mathbb Z^{2}$.

The manifold $M_{n}$ is the total space of the flat bundle
\[
\Sigma_{g}\hookrightarrow M_{n}\to T^{2}.
\]
Pulling this bundle back to the universal cover $\mathbb R^{2}\to T^{2}$ gives a trivial bundle over $\mathbb R^{2}$. The universal cover of $M_{n}$ is therefore diffeomorphic to
\[
\mathbb R^{2}\times \widetilde{\Sigma}_{g}.
\]
Since $g\ge 2$, $\widetilde{\Sigma}_{g}$ is contractible. Hence $M_{n}$ is aspherical.

Choose a symplectic basis
\[
a_{1},b_{1},\dots,a_{g},b_{g}
\]
of $H_{1}(\Sigma_{g};\mathbb Z)$ such that
\[
[\alpha]=a_{1},
\qquad
[\beta]=a_{2}.
\]
The action of a Dehn twist on first homology is given by
\[
(\tau_{\gamma})_{*}(x)=x+(x\cdot [\gamma])[\gamma],
\]
where $\cdot$ denotes the algebraic intersection pairing. Hence
\[
(\tau_{\alpha}^{n})_{*}(b_{1})=b_{1}+n a_{1},
\qquad
(\tau_{\alpha}^{n})_{*}(x)=x
\quad\text{for all other basis vectors},
\]
and
\[
(\tau_{\beta})_{*}(b_{2})=b_{2}+a_{2},
\qquad
(\tau_{\beta})_{*}(x)=x
\quad\text{for all other basis vectors}.
\]
Thus the coinvariant quotient of $H_{1}(\Sigma_{g};\mathbb Z)$ is obtained by imposing the relations
\[
n a_{1}=0,
\qquad
a_{2}=0.
\]
Therefore
\[
H_{1}(\Sigma_{g};\mathbb Z)_{\mathbb Z^{2}}
\cong
\mathbb Z^{2g-2}\oplus \mathbb Z/n.
\]
By \ref{lem:ab},
\[
H_{1}(M_{n};\mathbb Z)
\cong
H_{1}(\Sigma_{g};\mathbb Z)_{\mathbb Z^{2}}\oplus \mathbb Z^{2}
\cong
\mathbb Z^{2g}\oplus \mathbb Z/n.
\]

If $m\ne n$, then the torsion subgroups of $H_{1}(M_{m};\mathbb Z)$ and $H_{1}(M_{n};\mathbb Z)$ have different orders. Hence the first homology groups are not isomorphic. Therefore the fundamental groups are not isomorphic. Since the manifolds are aspherical, a homotopy equivalence between $M_{m}$ and $M_{n}$ would induce an isomorphism of their fundamental groups. Thus the manifolds are pairwise not homotopy equivalent.
\end{proof}

\subsection{A family with no Anosov diffeomorphisms}

\begin{proof}[Proof of Corollary \ref{cor:anosov}]
The second generator of $\mathbb Z^{2}$ acts trivially. Hence the flat bundle over
\[
T^{2}=S^{1}\times S^{1}
\]
is the product of the mapping torus of $A_{n}$ with the second circle. Thus
\[
X_{n}\cong N_{n}\times S^{1}.
\]

Since $A_{n}$ is a nontrivial unipotent element of $SL(2,\mathbb Z)$, its mapping torus $N_{n}$ is a compact quotient of the three-dimensional Heisenberg group. Therefore $X_{n}$ is a compact nilmanifold modeled on $Nil^{3}\times \mathbb R$.

We show that $X_{n}$ admits no Anosov diffeomorphism. By the Franks--Manning theory for Anosov diffeomorphisms on nilmanifolds, an Anosov diffeomorphism on a nilmanifold is topologically conjugate to a hyperbolic automorphism of the underlying nilmanifold. Thus it suffices to observe that the nilpotent Lie algebra of $Nil^{3}\times \mathbb R$ admits no hyperbolic lattice-preserving automorphism.

Let $\mathfrak n$ be the Lie algebra with basis $X,Y,Z,T$ and only nontrivial bracket
\[
[X,Y]=Z.
\]
The center is spanned by $Z,T$. If $\Phi$ is an automorphism of $\mathfrak n$, then the induced map on $\mathfrak n/[\mathfrak n,\mathfrak n]$ restricts on the span of $X,Y$ to some invertible $2\times 2$ block $C$. Since
\[
\Phi(Z)=\Phi([X,Y])=[\Phi(X),\Phi(Y)],
\]
the coefficient of $Z$ in $\Phi(Z)$ is $\det C$. For a lattice-preserving automorphism, this determinant is $\pm 1$. Hence the induced map on the center has an eigenvalue of modulus $1$. Therefore no such automorphism is hyperbolic. It follows that $X_{n}$ admits no Anosov diffeomorphism.

Finally,
\[
\pi_{1}(X_{n})
\cong
(\mathbb Z^{2}\rtimes_{A_{n}}\mathbb Z)\times \mathbb Z.
\]
Applying \ref{lem:ab} to $\mathbb Z^{2}\rtimes_{A_{n}}\mathbb Z$, one obtains
\[
H_{1}(N_{n};\mathbb Z)\cong \mathbb Z^{2}\oplus \mathbb Z/n.
\]
The extra $S^{1}$-factor contributes one more free generator, so
\[
H_{1}(X_{n};\mathbb Z)\cong \mathbb Z^{3}\oplus \mathbb Z/n.
\]
The torsion subgroups distinguish the homotopy types.
\end{proof}


\begin{thebibliography}{99}

\bibitem{Rokhlin} V. A. Rokhlin, ``On the fundamental ideas of measure theory'', Mat. Sb. (1949).

\bibitem{Sinai1968} Ya. G. Sinai, ``Markov partitions and $C$-diffeomorphisms'', Funkts. Anal. Prilozh. 2 (1968), 64--89.

\bibitem{Bowen1970} R. Bowen, ``Markov partitions for Axiom A diffeomorphisms'', Amer. J. Math. 92 (1970), 725--747.

\bibitem{Ruelle1978} D. Ruelle, \emph{Thermodynamic Formalism}, Addison-Wesley, 1978.

\bibitem{Young1998} L. S. Young, ``Statistical properties of dynamical systems with some hyperbolicity'', Ann. of Math. 147 (1998), 585--650.

\bibitem{Dolgopyat1998} D. Dolgopyat, ``On decay of correlations in Anosov flows'', Ann. of Math. 147 (1998), 357--390.

\bibitem{Liverani1995} C. Liverani, ``Decay of correlations for piecewise expanding maps'', J. Stat. Phys. 78 (1995), 1111--1129.

\bibitem{AvilaGouezel} A. Avila and S. Gouëzel, ``Small eigenvalues of the Laplacian for algebraic measures in dimension two'', Ann. Sci. Éc. Norm. Supér. 46 (2013), 735--775.

\bibitem{BurnsDolgopyat} K. Burns and D. Dolgopyat, ``Exponential mixing for time changes of geodesic flows'', J. Mod. Dyn. 6 (2012), 497--525.

\bibitem{FathiLaudenbachPoenaru} A. Fathi, F. Laudenbach, and V. Poénaru, \emph{Travaux de Thurston sur les surfaces}, Astérisque 66-67 (1979).

\bibitem{BestvinaHandel} M. Bestvina and M. Handel, ``Train tracks for surface homeomorphisms'', Topology 34 (1995), 109--140.

\bibitem{Anosov67} D. V. Anosov, ``Geodesic flows on closed Riemannian manifolds of negative curvature'', Proc. Steklov Inst. Math. 90 (1967).

\bibitem{Baladi2000} V. Baladi, \emph{Positive Transfer Operators and Decay of Correlations}, World Scientific, 2000.

\bibitem{ChernovMarkarian} N. Chernov and R. Markarian, \emph{Chaotic Billiards}, AMS, 2006.

\bibitem{KatokHasselblatt} A. Katok and B. Hasselblatt, \emph{Introduction to the Modern Theory of Dynamical Systems}, Cambridge Univ. Press, 1995.

\bibitem{PollicottSharp} M. Pollicott and R. Sharp, ``Exponential error terms for growth functions on negatively curved surfaces'', Amer. J. Math. 116 (1994), 1329--1371.

\bibitem{Pesin} Y. B. Pesin, \textit{Dimension Theory in Dynamical Systems}, University of Chicago Press, Chicago, 1997.

\bibitem{Smale} S. Smale, Differentiable dynamical systems, \textit{Bull. Amer. Math. Soc.} \textbf{73} (1967), 747--817.

\bibitem{BrinPesin} M. Brin and Y. B. Pesin, Partially hyperbolic dynamical systems, \textit{Izv. Akad. Nauk SSSR Ser. Mat.} \textbf{38} (1974), 170--212.

\bibitem{Dolgopyat} D. Dolgopyat, On decay of correlations in Anosov flows, \textit{Ann. of Math. (2)} \textbf{147} (1998), no. 2, 357--390.

\bibitem{FranksManning} J. Franks and A. Manning, Anosov diffeomorphisms on infranilmanifolds, \textit{Topology} \textbf{10} (1971), no. 3, 253--262.

\bibitem{RuelleSullivan} D. Ruelle and D. Sullivan, ``Currents, flows and diffeomorphisms'', Topology 14 (1975), 319--327.

\bibitem{Sinai1972} Ya. G. Sinai, ``Gibbs measures in ergodic theory'', Russian Math. Surveys 27 (1972), 21--69.

\bibitem{Bowen1975} R. Bowen, \emph{Equilibrium States and the Ergodic Theory of Anosov Diffeomorphisms}, Springer Lecture Notes in Math. 470, 1975.

\bibitem{ParryPollicott} W. Parry and M. Pollicott, ``Zeta functions and the periodic orbit structure of hyperbolic dynamics'', Astérisque 187-188 (1990).

\bibitem{Stallings} J. Stallings, ``On fibering certain 3-manifolds'', Topology of 3-manifolds and related topics (1962).

\bibitem{Ehresmann} C. Ehresmann, ``Les connexions infinitésimales dans un espace fibré différentiable'', Colloque de Topologie, Bruxelles (1950).

\bibitem{FarbMargalit} B. Farb and D. Margalit, \emph{A Primer on Mapping Class Groups}, Princeton Mathematical Series, vol. 49, Princeton University Press, Princeton, 2011.

\bibitem{Franks1969} J. Franks, \emph{Anosov diffeomorphisms on tori}, Trans. Amer. Math. Soc. \textbf{145} (1969), 117--124.

\bibitem{FranksSurvey} J. Franks, \emph{Anosov diffeomorphisms}, in \emph{Global Analysis}, Proc. Sympos. Pure Math., vol. 14, American Mathematical Society, Providence, RI, 1970, pp. 61--93.

\bibitem{HammerlindlPotrieSurvey} A. Hammerlindl and R. Potrie, \emph{Partial hyperbolicity and classification: a survey}, in \emph{Proceedings of the International Congress of Mathematicians 2018}, vol. III, World Scientific, 2018, pp. 1879--1908.

\bibitem{Hatcher} A. Hatcher, \emph{Algebraic Topology}, Cambridge University Press, Cambridge, 2002.

\bibitem{Hillman} J. A. Hillman, \emph{Four-Manifolds, Geometries and Knots}, Geometry \& Topology Monographs, vol. 5, Geometry \& Topology Publications, Coventry, 2002.

\bibitem{Hirsch} M. W. Hirsch, \emph{Anosov maps, polycyclic groups and homology}, Topology \textbf{10} (1971), 177--183.

\bibitem{Husemoller} D. Husemoller, \emph{Fibre Bundles}, third edition, Graduate Texts in Mathematics, vol. 20, Springer-Verlag, New York, 1994.

\bibitem{ManningNil} A. Manning, \emph{Anosov diffeomorphisms on nilmanifolds}, Proc. Amer. Math. Soc. \textbf{38} (1973), 423--426.

\bibitem{Manning1974} A. Manning, \emph{There are no new Anosov diffeomorphisms on tori}, Amer. J. Math. \textbf{96} (1974), 422--429.

\bibitem{Neofytidis} C. Neofytidis, \emph{Anosov diffeomorphisms on Thurston geometric \(4\)-manifolds}, Geom. Dedicata \textbf{213} (2021), 251--266.

\bibitem{RodriguezHertzSurvey} F. Rodriguez Hertz, M. A. Rodriguez Hertz, and R. Ures, \emph{A survey of partially hyperbolic dynamics}, in \emph{Partially Hyperbolic Dynamics, Laminations, and Teichm\"uller Flow}, Fields Institute Communications, vol. 51, American Mathematical Society, Providence, RI, 2007, pp. 35--87.

\bibitem{CandelConlon} A. Candel and L. Conlon, \emph{Foliations I, II}, Graduate Studies in Mathematics, vol. 23, 60, American Mathematical Society, 2000, 2003.

\bibitem{HaefligerGroupoids} A. Haefliger, \emph{Groupoids and foliations}, in \emph{Groupoids in Analysis, Geometry, and Physics}, Contemp. Math. 282, Amer. Math. Soc., 2001, pp. 83--100.

\bibitem{Molino} P. Molino, \emph{Riemannian Foliations}, Progress in Mathematics, vol. 73, Birkh\"auser, 1988.

\end{thebibliography}
\end{document}